\documentclass[12pt]{extarticle}
\usepackage{graphicx} 
\usepackage[margin=0.8in]{geometry}
\usepackage{amsmath,comment}
\usepackage{amsthm}
\usepackage{amssymb}
\usepackage{mathtools}
\usepackage{xcolor}
\usepackage{hyperref}
\usepackage{ulem}
\usepackage{enumitem}
\usepackage[capitalise]{cleveref}

\newtheorem{theorem}{Theorem}[section]
\newtheorem{lemma}[theorem]{Lemma}
\newtheorem{claim}[theorem]{Claim}

\newtheorem{conjecture}[theorem]{Conjecture}
\newtheorem{observation}[theorem]{Observation}
\newtheorem*{observation*}{Observation}
\newtheorem{proposition}[theorem]{Proposition}

\newtheorem{question}[theorem]{Question}
\newtheorem*{question*}{Question}
\newtheorem{innerdefinition}{Definition}

\newenvironment{definition*}
  {
   \innerdefinition}
  {\endinnerdefinition}

\theoremstyle{definition}
\newtheorem{defn}[theorem]{Definition}

\theoremstyle{remark}
\newtheorem*{remark}{Remark}

\newenvironment{poc}{\begin{proof}[Proof of claim]}{\end{proof}}

\newcommand*{\N}{\mathbb{N}}
\newcommand*{\Z}{\mathbb{Z}}
\newcommand*{\R}{\mathbb{R}}

\newcommand*{\cC}{\mathcal{C}}
\newcommand*{\cF}{\mathcal{F}}

\newcommand*{\cT}{\mathcal{T}}
\newcommand*{\mis}{\mathrm{mis}}
\newcommand*{\MIS}{\mathrm{MIS}}
\newcommand*{\msf}{\mathrm{msf}}
\newcommand*{\MSF}{\mathrm{MSF}}
\newcommand*{\spa}{\mathrm {span}}
\newcommand*{\ri}{\MakeUppercase{\romannumeral 1}}
\newcommand*{\rii}{\MakeUppercase{\romannumeral 2}}
\newcommand*{\riii}{\MakeUppercase{\romannumeral 3}}

\usepackage{todonotes}

\usepackage{etoolbox}
\patchcmd{\thebibliography}{\labelsep}{\labelsep\itemsep=2pt \parsep=1.5pt \relax}{}{}

\title{Infinitely many groups exhibiting intermediate growth in maximal sum-free sets}
\author{
J\'ozsef Balogh\thanks{Department of Mathematics, University of Illinois Urbana-Champaign, Urbana, IL, USA, and Extremal Combinatorics and Probability Group (ECOPRO), Institute for Basic Science (IBS), Daejeon, South Korea. E-mail: \texttt{jobal@illinois.edu}. Supported by NSF grants RTG DMS-1937241, FRG DMS-2152488, the Arnold O. Beckman Research Award (UIUC Campus Research Board RB 24012), the Simons Fellowship,
Simons Collaboration grant [SFI-MPS-TSM-00013107, JB], and the Institute for Basic Science (IBS-R029-C4).
} 
\and 
Ramon I. Garcia\thanks{Department of Mathematics, University of Illinois Urbana-Champaign, Urbana, IL, USA. E-mail: \texttt{rig2@illinois.edu}. Supported by NSF grant RTG DMS-1937241 and UIUC Campus Research Board RB 24012.
} 
\and
Hong Liu\thanks{Extremal Combinatorics and Probability Group (ECOPRO), Institute for Basic Science (IBS), Daejeon, South Korea. E-mail: \texttt{hongliu@ibs.re.kr}. Supported by the Institute for Basic Science (IBS-R029-C4).}
\and
Ningyuan Yang\thanks{School of Mathematical Sciences, Fudan University, Shanghai, China, and Extremal Combinatorics and Probability Group (ECOPRO), Institute for Basic Science (IBS), Daejeon, South Korea. E-mail: \texttt{nyyang23@m.fudan.edu.cn}.}
}
\date{\today}

\begin{document}

\maketitle

\begin{abstract}
Given an Abelian group $G$, denote $\mu(G)$ the size of its largest sum-free subset and $f_{\max}(G)$ the number of maximal sum-free sets in $G$. Confirming a prediction by Liu and Sharifzadeh, we prove that all even-order $G\ne \Z_2^k$ have exponentially fewer maximal sum-free sets than $\Z_2^k$, i.e., $f_{\max}(G) \leq 2^{(1/2-c)\mu(G)}$, where $c > 10^{-64}$.
    
We construct an infinite family of Abelian groups $G$ with intermediate growth in the number of maximal sum-free sets, i.e., with  $
    2^{(\frac{1}{2}+c)\mu(G)}\leq f_{\max}(G) \leq 3^{(\frac{1}{3}-c)\mu(G)}
    $, where $c=10^{-4}$. This disproves a conjecture of Liu and Sharifzadeh and also answers a question of Hassler and Treglown in the negative.

Furthermore, we determine for every even-order group $G$, the number of maximal distinct  sum-free sets (where a distinct sum is $a+b= c$ with distinct $a,b,c$): it is 
$ 2^{(1/2+o(1))\mu(G)}$
    with the only exception being $G=\Z_2^k \oplus \Z_3$, when this function is  $3^{(1/3+o(1))\mu(G)}$, refuting a conjecture of Hassler and Treglown.

Our proofs rely on a container theorem due to Green and Ruzsa.
Another key ingredient is a sharp upper bound we establish on the number of maximal independent sets in graphs with given matching number, which interpolates
between the classical results of Moon and Moser, and Hujter and Tuza. A special case of our bound implies that every $n$-vertex graph with a perfect matching has at most $2^{n/2}$ maximal independent sets, resolving another conjecture of  Hassler and Treglown.
\end{abstract}

\section{Introduction}

Sum-free subsets of the integers or Abelian groups have been widely studied in additive combinatorics. Let $[n] \coloneqq \{1,\dots,n\}$. We call a triple $\{x,y,z\}$ a {\it Schur triple} if $x + y = z$. A set $S \subseteq [n]$ is {\it sum-free} if $S$ contains no Schur triple. Denote by $f(n)$ the number of sum-free subsets of $[n]$, and by $f_{\max}(n)$ the number of maximal sum-free sets of $[n]$. Note that all subsets of the odd integers and $\{\lfloor n/2\rfloor+1,\dots,n\}$ are sum-free, hence $f(n) \geq 2^{n/2}$. In \cite{cameron1990number}, Cameron and Erd\H os conjectured that $f(n)/2^{n/2}$ tends to two different constants, depending on the parity of $n$, which was
confirmed independently by Green \cite{2003The} and Sapozhenko~\cite{SAPOZHENKO20084361}. In a subsequent paper, Cameron and Erd\H os \cite{camerd} posed the problem of determining $f_{\max}(n)$. This was resolved by the first and third authors, jointly with Sharifzadeh and Treglown \cite{J2015The,2018maxsumfree}, who showed that $f_{\max}(n)/2^{n/4}$ tends to four constants depending on $n$ modulo $4$.

The problem of counting sum-free sets in Abelian groups is more intricate. Let $\mu(G)$ denote the maximum size of a sum-free subset of an Abelian group $G$. The study of $\mu(G)$ dates back to the work of Diananda and Yap~\cite{1969Maximal} in 1969, and it  was finally determined for every group by Green and Ruzsa~\cite{2005Sum}. 

\begin{theorem}[\cite{1969Maximal,2005Sum}]\label{thm:mu}
Let $G$ be an Abelian group of order $n$. Then $G$ falls into one of the following types, and $\mu(G)$ is given as follows.
\begin{enumerate}[label=(\roman*)]
\item \label{typeI} If $n$ is divisible by a prime $p\equiv 2\pmod{3}$, let $p$ be the smallest such prime; then $G$ is of type \ri($p$)\ and $\mu(G) = (1/3+1/3p)n$.
\item \label{typeII} If $n$ is not divisible by any prime $p\equiv 2\pmod{3}$ but $n$ is a multiple of $3$, then $G$ is of type \rii\ and $\mu(G) = n/3$.
\item \label{typeIII} Otherwise, $G$ is of type \riii\ and $\mu(G) = (1/3-1/3m)n$,
where $m$ is the exponent of $G$.
\end{enumerate}
\end{theorem}

Denote by $f(G)$ and $f_{\max}(G)$  the number of sum-free subsets and the number of maximal sum-free subsets of $G$, respectively. Green and Ruzsa \cite{2005Sum} proved that $f(G) = 2^{(1+o(1))\mu(G)}$. Let us define the {\it growth} of $f_{\max}(G)$ to be the smallest $g \in \R$ such that
\[
f_{\max}(G) \leq g^{(1+o(1))\mu(G)}.
\]
Balogh, Liu, Sharifzadeh and Treglown \cite{2018maxsumfree} proved that, similarly to the integer setting, $f_{\max}(G)$ is exponentially smaller than $f(G)$ for all finite Abelian groups $G$: $f(G)_{\max} \le 3^{(1/3+o(1))\mu(G)}$. 
In other words, $f_{\max}(G)$ has growth at most $3^{1/3}$ for every group $G$.

Recall that in the integer setting, the growth of $f_{\max}([n])$ is $2^{1/2}$. It was asked in \cite{2018maxsumfree} whether the growth of  $f_{\max}(G)$ is also at most $2^{1/2}$; they showed that it is the case for $\Z_2^k$. It was also conjectured~\cite{2018maxsumfree} that there are infinitely many groups for which the growth of $f_{\max}$ is strictly less than $2^{1/2}$. For the former question, Liu and Sharifzadeh~\cite{2021GroupsLiu} gave a negative answer, showing that there are infinitely many groups for which $f_{\max}$ are of maximum growth $3^{1/3}$. 
 They settled the latter conjecture affirmatively, showing that for almost all even-order groups, the growth of $f_{\max}$ are at most $2^{1/2-c}$ for some absolute constant $c>0$. In~\cite{2021GroupsLiu}, it was mentioned that \textit{``Our result suggests that $\Z_2^k$ might be the only exception among all even order groups achieving the bound $2^{(1/2+o(1))\mu(G)}$.''}

Our first result confirms this belief.

\begin{theorem}\label{thm:even}
    There exists a constant $c > 10^{-64}$ and an integer $n_0$ such that the following holds. Let $G$ be an Abelian group with even order $n>n_0$. If $G\ne \Z_2^k$, then
    \[
    f_{\max}(G) \leq 2^{(1/2-c)\mu(G)}.
    \]
\end{theorem}

In~\cite{2021GroupsLiu}, Liu and Sharifzadeh made an even bolder conjecture that the same is true for all type \ri \ groups.

\begin{conjecture}[\cite{2021GroupsLiu}]\label{conj:typeI-few}
    All type I groups $G$ except $\Z_2^k$ have exponentially fewer maximal sum-free sets than $2^{\mu(G)/2}$.
\end{conjecture}

To support this, they showed that this is the case for type \ri(5) groups. Recently, a related question was raised by Hassler and Treglown \cite{hassler2025notessumfreesetsAbelian}. In light of no examples of intermediate growth between $2^{1/2}$ and $3^{1/3}$, they asked the following question.

\begin{question}[\cite{hassler2025notessumfreesetsAbelian}]\label{qs:n*}
    Given an Abelian group $G$ of order $n$, is it true that either $f_{\max}(G) = 3^{\mu(G)/3 + o(n)}$ or $f_{\max}(G) \leq 2^{\mu(G)/2 + o(n)}$?
\end{question}

Our next result provides explicit constructions of infinitely many groups with intermediate growth in the number of maximal sum-free sets. This gives negative answers to both~\cref{conj:typeI-few} and~\cref{qs:n*}. 

\begin{theorem}\label{thm:growth}
    Let $c=10^{-4}$. There are infinitely many Abelian groups $G$ of type I such that 
    \[
    2^{(\frac{1}{2}+c)\mu(G)}\leq f_{\max}(G) \leq 3^{(\frac{1}{3}-c)\mu(G)}.
    \]
\end{theorem}

In \cite{hassler2025notessumfreesetsAbelian}, Hassler and Treglown initiated the study of distinct sum-free subsets of Abelian groups. A triple $\{x,y,z\}$ is a {\it distinct Schur triple} if $x + y = z$ and $x,y,z$ are distinct. A set $S \subseteq G$ is {\it distinct sum-free} if $S$ contains no distinct Schur triple. Denote by $\mu^*(G)$ the maximum size of distinct sum-free subsets of $G$, by $f^*(G)$ the number of distinct sum-free sets of $G$, and by $f^*_{\max}(G)$ the number of maximal distinct sum-free sets of $G$.

Hassler and Treglown~\cite{2022MaximalHassler} observed that using the removal lemma of Green \cite{2005A}, one has $\mu(G) \leq \mu^*(G) \leq \mu(G)+o(n)$ and $f(G) \leq f^*(G) \leq 2^{o(n)}\cdot f(G)$. By a method similar to   the proof of \cite{2018maxsumfree}[Proposition 5.1], one also knows that $f^*_{\max}(G) \leq 3^{\mu^*(G)/3+o(n)} = 3^{\mu(G)/3+o(n)}$. Hassler and Treg\-lown~\cite{hassler2025notessumfreesetsAbelian} conjectured that $f^*_{\max}$ of all type \ri\ Abelian groups have growth $2^{1/2}$.

\begin{conjecture}[\cite{hassler2025notessumfreesetsAbelian}]\label{conj:*}
    For all type \ri\ groups $G$, $f^*_{\max}(G) = 2^{(1/2+o(1))\mu(G)}$.
\end{conjecture}

They examined even-order groups in depth, giving a construction to prove the lower bound $f^*_{\max}(G) \geq 2^{(n-2)/4}$. For the upper bound, they showed the validity of their conjecture for almost all even-order groups.

We disprove~\cref{conj:*} with counterexamples of type \ri($p$)  for $p=2$ and every $p\ge 23$. We show that a relaxed version of the conjecture is true, that is, it is true for all but one  even-order group.

\begin{theorem}\label{thm:even-distinct}
    For every even-order group $G$, 
    \[
    f_{\max}^*(G) = 2^{(1/2+o(1))\mu(G)},
    \]
    with the only exception $f_{\max}^*(\Z_2^k \oplus \Z_3) = 3^{(1+o(1))\mu(G)/3}$.

    For every $p\ge 23$ with $p \equiv 2\ (\text{mod}\ 3)$, there are type \ri($p$) groups with $f_{\max}^*(G)$ exponentially larger than $2^{\mu(G)/2}$.
\end{theorem}

A key tool in our proofs is an optimal bound on the number of maximal independent sets in graphs with given matching number. Denote by $\MIS(\Gamma)$ the family of maximal independent sets in a graph $\Gamma$, and $\mis(\Gamma)=|\MIS(\Gamma)|$. Moon and Moser~\cite{1965On} proved that for every $n$-vertex graph $\Gamma$, $\mis(\Gamma) \leq 3^{n/3}$, with equality if and only if $\Gamma$ is the vertex disjoint union of $n/3$ triangles. Upper bounds on $\mis(\Gamma)$ have been provided for various classes of graphs, such as connected graphs~\cite{1987MISFuredi, 1988MISGriggs} and trees~\cite{1988MISSagan, 1986MISWilf}.
An important extension is due to 
 Hujter and Tuza~\cite{1993HujterTuza}. They proved that for every $n$-vertex triangle-free graph $\Gamma$, $\mis(\Gamma) \leq 2^{n/2}$,  
with equality if and only if $\Gamma$ is the vertex disjoint union of $n/2$ edges.

   Recently, Kahn and Park~\cite{kahn2020stability} proved  stability versions of results of Moon-Moser and Hujter-Tuza. They improved the exponents by $\Omega_{\delta}(n)$ in case $\Gamma$ contains at most $(1/3-\delta)n$ vertex disjoint triangles with no edges between the triangles (i.e., the {\it induced triangle matching} number is at most $(1/3-\delta)n$), and in case when $\Gamma$ is triangle-free and the induced matching number is at most  $(1/2-\delta)n$,  for some small constant $\delta>0$. Kahn and Park's results were subsequently strengthened by Palmer and Patk\'os~\cite{2023MISPalmer}. In particular, they proved that if an $n$-vertex graph does not contain an induced triangle matching of size $t+1$, then it has at most $3^t\cdot 2^{m/2}$ maximal independent sets, where $0\le t\le n/3$ and $m=n-3t$ is even.

In \cite{hassler2025notessumfreesetsAbelian}, Hassler and Treglown proposed the following conjecture.

\begin{conjecture}[\cite{hassler2025notessumfreesetsAbelian}]\label{conj:mis}
    Let $\Gamma$ be an $n$-vertex graph that contains a perfect matching. Then
    \[
    \mis(\Gamma) \leq 2^{n/2}.
    \]
\end{conjecture}

We resolve this conjecture in a stronger form by establishing a thorough relationship between the matching number and the number of maximal independent sets. We denote by $D_6$ the graph which consists of two vertex disjoint triangles joined by an edge.

\begin{theorem}\label{thm:matching}
    Let $\Gamma$ be an $n$-vertex graph with matching number $\nu(\Gamma)=m$.
    \begin{itemize}
        \item[(i)] If $m \leq \frac{n}{3}$, then
$$\mis(\Gamma) \leq  3^m.$$
Equality holds if and only if each component of $\Gamma$ is a $K_3$ or an isolated vertex.

        \item[(ii)] If $m > \frac{n}{3}$, then
$$\mis(\Gamma) \leq 2^{3m-n}\cdot 3^{n-2m}.$$
Equality holds if and only if each component of $\Gamma$ is a $K_2$, $K_3$, $K_4$ or $D_6$.
    \end{itemize}
 \end{theorem}
We remark that both bounds in~\cref{thm:matching} are valid for all values of $m$. Observe, that by the Moon-Moser~\cite{1969Maximal} and Hujter-Tuza~\cite{1993HujterTuza} Theorems, the appearance of $K_2$ and $K_3$ is expected in the extremal examples. For $K_4$ and $D_6$, observe that $6K_2$, $3K_4$ and $2D_6$ all contribute the same, $12$, to the number of vertices, and each has a perfect matching, and each has $64$ maximal independent sets. Part (i) of \cref{thm:matching} was proved by Shi, Tu and Wang \cite{shi2025maximalindependentsetsgraphs}, additionally,  they also proved several results on connected graphs with given matching number. For  completeness, we include a different proof.

The bulk of the proof of Theorem~\ref{thm:matching} is to handle part (ii). Note that the special case   $m=n/2$ above verifies~\cref{conj:mis}. To remember~\cref{thm:matching}, a graph that is helpful to keep in mind is the vertex disjoint union of $t$ triangles and $s$ edges. In this case, $n=2s+3t$ and $m=s+t$, and the number of maximal independent sets is exactly $2^s3^t=2^{3m-n}\cdot 3^{n-2m}$. Thus,~\cref{thm:matching} interpolates smoothly between the classical bounds given by Moon-Moser and Hujter-Tuza. Considering the wide applicability of the Moon-Moser and Hujter-Tuza Theorems, we anticipate that~\cref{thm:matching} will serve as a fundamental tool with applications beyond this work.

Comparing to the result of Palmer and Patk\'os~\cite{2023MISPalmer}, the advantage of~\cref{thm:matching} is that it is often difficult to estimate the \textit{induced} triangle matching number of a graph, whereas getting bounds on the matching number is a more manageable task. For instance, this is how we achieve the upper bound in~\cref{thm:growth}. Indeed, we use $\mis(\Gamma)$ of certain auxiliary graph $\Gamma$ to bound $f_{\max}(G)$, and then we identify a large matching in $\Gamma$ to invoke~\cref{thm:matching}. Also, for some graphs, \cref{thm:matching} gives a stronger bound. For instance, let $\Gamma$ be vertex disjoint union of $n/6$ triangles and $n/8$ $K_4$'s. Then  
$\mis(\Gamma)= 3^{n/6}\cdot 2^{n/4}$, which coincides with our upper bound in \cref{thm:matching}, while the bound by Palmer and Patk\'os gives only $3^{7n/24}\cdot 2^{n/16}\gg 3^{n/6}\cdot 2^{n/4} $, as $ \Gamma$ contains an  induced triangle matching of size $7n/24$.

Finally, we remark that~\cref{thm:matching} as stated does not suffice for all of our applications. For instance, for~\cref{thm:even,thm:even-distinct}, we in fact need several stability versions of~\cref{thm:matching}, which provide further (exponential) improvements when the graph is far away from the extremal examples (see~\cref{lem:degree4,lem:tl,lem:c4}).

\paragraph{Organization.} The rest of the paper is organized as follows. We first prove~\cref{thm:matching} in~\cref{proofthmmat}, then~\cref{thm:growth} in~\cref{sec:growth},~\cref{thm:even-distinct} in~\cref{sec:distinct}, and~\cref{thm:even}  in~\cref{sec:even-order}, respectively.

\section{Maximal independent sets given matching number}\label{proofthmmat}

In this section we shall prove \cref{thm:matching} and some stability extensions (\cref{lem:degree4,lem:tl,lem:c4}) that will be useful later. We split the proof of \cref{thm:matching} into two parts. The first part follows from the statement below.

\begin{theorem}\label{thm:smallermatchingnumber}
    Suppose $\Gamma$ is an $n$-vertex graph with matching number $\nu(\Gamma) = m$. If $\Gamma$ is not a vertex-disjoint union of $K_3$'s and isolated vertices, then $\mis(\Gamma) < 3^m$.
\end{theorem}

The second part of \cref{thm:matching} follows from the following theorem.

\begin{theorem}\label{thm:biggermatchingnumber}
    Suppose that $\Gamma$ is an $n$-vertex graph with matching number $\nu(\Gamma) = m$. If $\Gamma$ is not a vertex-disjoint union of $K_2$'s, $K_3$'s, $K_4$'s and $D_6'$s, then $\mis(\Gamma) < 2^{3m-n}3^{n-2m}$.
\end{theorem}

Note that in~\cref{thm:smallermatchingnumber,thm:biggermatchingnumber} we do not have restriction on how large $\nu(\Gamma)$ is. 

\medskip

To establish these bounds, we start with some notation and observations. For a set $X$ and $x\in X$, we write $X-x\coloneqq X\setminus\{x\}$.  For a graph $\Gamma$ and a vertex set $X$, we write $\Gamma-X\coloneqq\Gamma[V(\Gamma)\setminus X]$.
Denote by $\MIS(\Gamma;v)$ the family of maximal independent sets containing $v$, and by $\MIS(\Gamma;\overline v)=\MIS(\Gamma)\setminus \MIS(\Gamma;v)$ those not containing $v$, and let $\mis(\Gamma;v)$ and $\mis(\Gamma;\overline v)$ denote the cardinality of these two sets, respectively.

The lemma below collects some simple facts about how $\mathrm{mis}(\Gamma)$ changes by removal of vertices.
\begin{lemma}\label{lem:basic-mis-ineq}
    Let $\Gamma$ be a graph and $v$ be an arbitrary vertex of $\Gamma$. Then the following hold.
   
       (i) For every induced subgraph $\Gamma'\subseteq \Gamma$, we have $\mathrm{mis}(\Gamma')\le \mathrm{mis}(\Gamma)$.
       
       (ii) $\mathrm{mis}(\Gamma;v)= \mathrm{mis}(\Gamma-v-N(v))$.
       
       (iii) $\mathrm{mis}(\Gamma;\overline{v})\le  \mathrm{mis}(\Gamma-v)$. 
       
       (iv) $\mathrm{mis}(\Gamma;\overline{v}) \le  \sum_{u\in N(v)}\mathrm{mis}(\Gamma;u)$.
       
        (v) Let $uv$ be an edge. Then 
        $$\mathrm{mis}(\Gamma)\le \mathrm{mis}(\Gamma-v-N(v))+\mathrm{mis}(\Gamma-u-N(u))+\mathrm{mis}(\Gamma-v;\overline{u}).$$
\end{lemma}

\begin{proof}
    (i) For each $I' \in \MIS(\Gamma')$, there is at least one maximal independent set $I \in \MIS(\Gamma)$ such that $I'\subseteq I$. Furthermore, for $I',J' \in \MIS(\Gamma')$ with $I' \neq J'$, by the maximality of $I'$ and $J'$ in $\Gamma'$, there is no $I\in \MIS(\Gamma)$ such that $I',J'\subseteq I$. Hence $\mis(\Gamma') \leq \mis(\Gamma)$.

    (ii) Note that a set $I\in\MIS(\Gamma;v)$ if and only if $I \cap N(v) = \varnothing$ and $I-v \in \MIS(\Gamma-v-N(v))$. Thus $\MIS(\Gamma;v) = \MIS(\Gamma-v-N(v)).$

    (iii) If $I \in \MIS(\Gamma;\overline v)$, then $I \in \MIS(\Gamma-v)$, and thus $\MIS(\Gamma;\overline v) \subseteq \MIS(\Gamma-v)$.

    (iv) If $I \in \MIS(\Gamma;\overline v)$, then $I \cap N(v) \neq \varnothing$. We have $\MIS(\Gamma;\overline v) = \bigcup_{u\in N(v)}\MIS(\Gamma;u)$.
    
    (v) Using (the proof of) part (iii), if $uv$ is an edge in $\Gamma$, then we have
    \begin{align*}
    \MIS(\Gamma) &= \MIS(\Gamma;v) \cup \MIS(\Gamma;\overline v)\\ & \subseteq \MIS(\Gamma;v) \cup \MIS(\Gamma - v)\\
    &= \MIS(\Gamma;v) \cup \MIS(\Gamma - v;u) \cup \MIS(\Gamma - v;\overline u)\\ &= \MIS(\Gamma;v) \cup \MIS(\Gamma;u) \cup \MIS(\Gamma - v;\overline u).
    \end{align*}
    The conclusion then follows from part (ii).
\end{proof}

\subsection{Proof of~\cref{thm:smallermatchingnumber}}

We say that $\Gamma$ is {\it extremal} if each component of $\Gamma$ is a $K_3$ or an isolated vertex. It is easy to see that $\MIS(\Gamma) = 3^m$ for an extremal graph $\Gamma$. Let $\Gamma$ be a counterexample with minimum number of vertices. Then $\Gamma$ is not extremal, and $\Gamma$ has at least one edge as the edgeless graph is extremal, giving that $m\geq1$. The only connected graphs with matching number 1 are $K_3$ and stars. Hence, when $m=1$, since $\Gamma$ is not extremal, it is a disjoint union of a star and isolated vertices, and $\mis(\Gamma) = 2 < 3^1$. Hence, we may assume that $m\ge 2$.
    
    Choose a maximum matching $M$ and $uv \in E(M)$. Note that each of $\Gamma-u-N(u)$, $\Gamma-v-N(v)$ and $\Gamma-v-u$ has matching number at most $m-1$. By the minimality of $\Gamma$ we have 
    \begin{equation}\label{eq:type-A}
         \mis(\Gamma-u-N(u)),\ \mis(\Gamma-v-N(v)),\ \mis(\Gamma-u-v) \leq 3^{m-1}.
    \end{equation}
    Now \cref{lem:basic-mis-ineq}(iii,v) gives that
    \[
    \mis(\Gamma) \leq \mis(\Gamma-u-N(u)) + \ \mis(\Gamma-v-N(v)) + \ \mis(\Gamma-u-v) \leq 3^m.
    \]
    It remains to show that at least one of the three inequalities in \eqref{eq:type-A} is strict.
    
    Suppose all three equalities hold in \eqref{eq:type-A}. Then by the minimality of $\Gamma$, each of the three graphs must be extremal and have matching number precisely $m-1$, and $\Gamma - u - v$ consists of $m-1$ disjoint $K_{3}$'s and isolated vertices. Suppose $u$ has an edge to any of these triangles, say to vertex $x$. Then $\Gamma - u - N(u)$ is an induced subgraph of $\Gamma - u - v - x$ which has strictly less than $3^{m-1}$ maximal independent sets, a contradiction. Similarly $v$ is also not adjacent to any vertex from these triangles. 
    
    Since the matching number of $\Gamma$ is $m$, it follows that $u,v$ with $N(u)$ and $N(v)$ together form a component $C$ isolated from these $m-1$ triangles, with matching number 1. Then the only way we can achieve $3^{m}$ maximal independent sets in $\Gamma$ is when $C$ is a triangle. But then $\Gamma$ is extremal, a contradiction.

\subsection{Proof of~\cref{thm:biggermatchingnumber}}

We say that $\Gamma$ is {\it extremal} if each component of $\Gamma$ is a $K_2$, $K_3$, $K_4$ or a $D_6$. An extremal graph $\Gamma$ satisfies that $\mis(\Gamma) = 2^{3m-n}3^{n-2m}$. Indeed, suppose the numbers of $K_2$, $K_3$, $K_4$ and $D_6$ are $a,b,c$ and $d$ respectively, then
\[
m=a+b+2c+3d,\quad n=2a+3b+4c+6d,\quad \mis(\Gamma)=2^{a+2c+3d}3^b=2^{3m-n}3^{n-2m}.
\]

Suppose for a contradiction that \cref{thm:biggermatchingnumber} is false, and let $\Gamma$ be   a counterexample  with minimum number of vertices. Choose a maximum matching $M$ in $\Gamma$, and an edge $uv \in E(M)$. Without loss of generality we may assume  $\deg(u) \leq \deg(v)$. Define $a(m,n) \coloneqq 2^{3m-n}3^{n-2m}$.

    Considering the $(n-2)$-vertex graph $\Gamma-u-v$ with $\nu(\Gamma-u-v) = m-1$, by \cref{lem:basic-mis-ineq}(iii) and the minimality of $\Gamma$, we have
\begin{equation}\label{eq:deleteedge}
    \mis(\Gamma-v;\overline u) \leq \mis(\Gamma-u-v) \leq a(m-1,n-2) = 2^{3m-n-1}3^{n-2m} = \frac{1}{2}a(m,n),
    \end{equation}
    and the second inequality is equality  if and only if $\Gamma-u-v$ is extremal.

\begin{claim}\label{cl:deletevertex}
    For any $x\in V(M)$ with $\deg(x)=d$, we have
    $$\mis(\Gamma-x-N(x)) \leq a(m-d,n-d-1) =  \frac{3^{d-1}}{2^{2d-1}}a(m,n),$$
    and equality holds if and only if $\Gamma-x-N(x)$ is extremal.     
\end{claim}
 \begin{poc}
     Observe that $\Gamma-x-N(x)$ is an $(n-d-1)$-vertex graph, and since $\{x\} \cup N(x)$ appears in at most $d$ edges of the matching $M$, we have $\nu(\Gamma-x-N(x)) \geq m-d$. As $a(m,n)$ is a decreasing function of $m$ (when $n$ is fixed), the minimality of $\Gamma$ gives     
    $$\mis(\Gamma-x-N(x)) \leq a(m-d,n-d-1) = 2^{3m-n-2d+1}3^{n-2m+d-1} = \frac{3^{d-1}}{2^{2d-1}}a(m,n),$$
    and equality holds if and only if $\Gamma-x-N(x)$ is extremal. 
 \end{poc}

    We split the rest of the proof into 7 cases depending on the degrees of $u$ and $v$.

    \paragraph{Case 1.} $\deg(u) = \deg(v) = 1$.\\ Then $\Gamma$ is a disjoint union of the edge $uv$ and $\Gamma-u-v$. As $\Gamma$ is not extremal, $\Gamma-u-v$ is not extremal either. Hence we have the strict inequality in~\eqref{eq:deleteedge}, and
    \[
    \mis(\Gamma) = 2\cdot \mis(\Gamma-u-v) < 2 \cdot \frac{1}{2}a(m,n) = a(m,n).
    \]

    \paragraph{Case 2.} $\deg(u) = 1$ and $\deg(v) \geq 2$.\\ In this case, $\MIS(\Gamma-v;\overline u) = \varnothing$.
    By \cref{cl:deletevertex} we know that
    \[
    \mis(\Gamma-v-N(v)) \leq \frac{3}{8}a(m,n).
    \]
    Note that in this case $N(u)=\{v\}$. Thus, by~\cref{lem:basic-mis-ineq}(v) and \eqref{eq:deleteedge}, we have
    \[
    \begin{split}
\mis(\Gamma) \leq \mis(\Gamma-u-v) + \mis(\Gamma-v-N(v)) + \mis(\Gamma-v;\overline u)\leq \frac{1}{2}a(m,n) + \frac{3}{8}a(m,n)  = \frac{7}{8}a(m,n).
    \end{split}
    \]
    Since this case will be used later, we analyze what happens when the equalities hold.  Note first that equality in~\eqref{eq:deleteedge} implies $\Gamma-u-v$ is extremal. Next, equality in~\cref{cl:deletevertex} implies that $N(v)=\{u,x\}$ with some $x \in V(M)$, and $\Gamma-u-v-x$ is extremal with $\nu(\Gamma-u-v-x) = m-2$. As both $\Gamma-u-v$ and $\Gamma-u-v-x$ are extremal, $x$ can only lie in a component of  $\Gamma-u-v$ that is either a $K_3$, $K_4$ or   $D_6$. 
    
    If $x$ is in a $K_3$ component in $\Gamma-u-v$, then $\nu(\Gamma-u-v-x)=\nu(\Gamma-u-v)=m-1$, a contradiction.
    
    If   $x$ is in a $K_4$ component in $\Gamma-u-v$, then $uvx$ is an induced $P_3$ and   $\deg(x) = 4$ in $\Gamma$. 

    If   $x$ is in a $D_6$ component in $\Gamma-u-v$, then $\deg(x) = 4$ in $\Gamma$.

\paragraph{Case 3.} $\deg(u) = \deg(v) = 2$. \\ Let $N(u)=\{v,x\}$ and $N(v)=\{u,y\}$, where  we allow $x=y$.

Suppose first that $x \in V(M)$ (the case when $y\in V(M)$ is similar). Recall that $\Gamma-u-v$ is an $(n-2)$-vertex graph with $\nu(\Gamma-u-v)=m-1$. Also, by~\cref{lem:basic-mis-ineq}(ii), $\mis(\Gamma-u-v;x)\le \mis(\Gamma-u-v-x-N(x))$. Applying \cref{cl:deletevertex} to vertex $x$ (which lies in a maximum matching $M-uv$) in  $\Gamma-u-v$, we get that
    \[
    \mis(\Gamma-u-v;x)\le \mis(\Gamma-u-v-x-N(x))\leq \frac{1}{2}a(m-1,n-2) = \frac{1}{4}a(m,n).
    \]
We also have $\mis(\Gamma-u-N(u)),\mis(\Gamma-v-N(v)) \leq \frac{3}{8}a(m,n)$ from \cref{cl:deletevertex}. Note that every $I\in \mis(\Gamma-v;\overline u)$ contains $x$ as $N(u)=\{v,x\}$. This, together with \cref{lem:basic-mis-ineq}(iii), implies that 
\begin{equation}\label{eq:-v-ubar}
    \mis(\Gamma-v;\overline u) \leq \mis(\Gamma-u-v;x) \leq \frac{1}{4}a(m,n).
\end{equation}
Then by \cref{lem:basic-mis-ineq}(v), we have
    \[
    \mis(\Gamma) \leq \frac{3}{8}a(m,n) + \frac{3}{8}a(m,n) + \frac{1}{4}a(m,n) = a(m,n).
    \]
    When the equalities hold, $\mis(\Gamma-u-v;x)$ attains $a(m-1,n-2)/2$ in \cref{cl:deletevertex}, hence $x$ has degree 1 (let $x \sim z$) in $\Gamma-u-v$ and $\Gamma-u-v-x-z$ is extremal with $\nu(\Gamma-u-v-x-z)=m-2$. Furthermore, $\mis(\Gamma-u-N(u)) = \mis(\Gamma-u-v-x)$ attains $a(m-2,n-3)$, hence $\Gamma-u-v-x$ is also extremal with $\nu(\Gamma-u-v-x)=m-2$. This implies that $z$ is in a $K_3$ component in $\Gamma-u-v-x$. 
    Indeed, if $z$ was part of a $K_2$ then
    $\Gamma-u-v-x-z$ would not be extremal, and 
    if $z$ was part of a $K_4$ or a $D_6$ then
$\nu(\Gamma-u-v-x)=m-1$, a contradiction.


    If $y \neq x$, then $\Gamma-u-v-y = \Gamma-v-N(v)$ is also extremal. Then $y\ne z$, otherwise $x$ would be an isolated vertex in $\Gamma-u-v-y$.  Now, $x \sim z$ in $\Gamma-u-v-y$, where $x$ has degree $1$ and $z$ has degree at least $2$, which is impossible in an extremal graph. Therefore, $x=y$ and $\{u,v,x\}$ is a $K_3$  in a $D_6$ component in $\Gamma$, where the other $3$ vertices form a $K_3$ component in the extremal graph $\Gamma-u-v-x$. This contradicts the fact that $\Gamma$ is not extremal.
    
    We may then assume that $x,y \notin V(M)$. Then by the maximality of $M$, we must have $x=y$. Note that $\Gamma-u-N(u)=\Gamma-u-v-x$ is an $(n-3)$-vertex graph with $\nu(\Gamma-u-N(u)) \geq m-1$. Hence by~\cref{lem:basic-mis-ineq}(ii),
    \[
    \mis(\Gamma;u) \leq \mis(\Gamma-u-v-x) \leq a(m-1,n-3)=2^{3m-n}3^{n-2m-1} = \frac{1}{3}a(m,n).
    \]
    Similarly, $\mis(\Gamma;v) \leq a(m,n)/3$. On the other hand,
    \[
    \mis(\Gamma-v;\overline u) \leq \mis(\Gamma-v;x)\le \mis(\Gamma-v-x-N(x)) \leq \mis(\Gamma-u-v-x) \leq \frac{1}{3}a(m,n),
    \]
    where the first inequality follows from $N(u)=\{v,x\}$ and the second and third ones follow from Lemma~\ref{lem:basic-mis-ineq}(ii) and (i) respectively.
    Putting  together, we also have $\mis(\Gamma) \leq a(m,n)$.

    When the equalities hold, $\mis(\Gamma-v;x) = a(m,n)/3 = a(m,n-1)/2$ and $\nu(\Gamma-u-v-x) = m-1$. Now $ux$ is an edge in a maximum matching $M-uv+ux$ of $\Gamma-v$, and by \eqref{eq:deleteedge} we know that $ux$ is an isolated edge in $\Gamma-v$ and $\Gamma-u-v-x$ is extremal. Therefore, $\{u,v,x\}$ is an isolated triangle in $\Gamma$, contradicting  the fact that $\Gamma$ is not extremal.
    
    \paragraph{Case 4.} $\deg(u) = 2$ and $\deg(v) \geq 3$.\\ Let $N(u)=\{v,x\}$. If $x\in V(M)$, then by~\cref{cl:deletevertex} and~\eqref{eq:-v-ubar} we have
    \[
    \mis(\Gamma-u-N(u)) \leq \frac{3}{8}a(m,n),\quad \mis(\Gamma-v-N(v)) \leq \frac{9}{32}a(m,n),\quad \mis(\Gamma-v;\overline u) \leq \frac{1}{4}a(m,n).
    \]
    Putting these together we have $\mis(\Gamma) \leq 29a(m,n)/32$.
    
    If $x\notin V(M)$, then we still have $\mis(\Gamma-u-N(u)),\mis(\Gamma-v;\overline u) \leq a(m,n)/3$ as in the previous case, and from \cref{cl:deletevertex} we have $\mis(\Gamma-v-N(v)) \leq 9a(m,n)/32$. 
Putting these together we have $\mis(\Gamma) \leq 91a(m,n)/96$.

    \paragraph{Case 5.} $\deg(u) \geq 3$ and $\deg(v) \geq 4$.\\ By \cref{lem:basic-mis-ineq}(iii) and \eqref{eq:deleteedge} we have
    \[
    \mis(\Gamma-v;\overline u) \leq \mis(\Gamma-u-v) \leq \frac{1}{2}a(m,n),
    \]
    and \cref{cl:deletevertex} implies
    \[
    \mis(\Gamma-u-N(u)) \leq \frac{9}{32}a(m,n),\quad \quad \mis(\Gamma-v-N(v)) \leq \frac{27}{128}a(m,n).
    \]
    Putting these together, by \cref{lem:basic-mis-ineq}(v) we have $\mis(\Gamma) \leq 127a(m,n)/128$.

\medskip

    The only remaining case is when $\deg(u) = \deg(v) = 3$. Since the choice of $uv$ in $M$ is arbitrary, from now on we may  suppose that every vertex in $M$ has degree $3$. Denote $N(u)=\{v,x,y\}$ and $N(v)=\{u,z,w\}$.

    \paragraph{Case 6.} Both $\{x,y\}$ and $\{z,w\}$ appear in at most one edge in $M$.\\ In this case $\Gamma-u-N(u)$ is an $(n-4)$-vertex graph with $\nu(\Gamma-u-N(u)) \geq m-2$. Hence
    \[
    \mis(\Gamma-u-N(u)) \leq a(m-2,n-4)= 2^{3m-n-2}3^{n-2m} = \frac{1}{4}a(m,n).
    \]
    Similarly, $\mis(\Gamma-v-N(v)) \leq a(m,n)/4$. Putting together with \eqref{eq:deleteedge}, we have from \cref{lem:basic-mis-ineq}(v) that $\mis(\Gamma) \leq a(m,n)$. 
    
    When the equalities of \eqref{eq:deleteedge} hold, we know that $\Gamma-u-v$ is extremal, and each maximal independent set of $\Gamma-u-v$ contains some neighbor of $u$. It implies that $\{x,y\}$ is an isolated edge in $\Gamma-u-v$. Analogously, $\{z,w\}$ is also an isolated edge. Depending on $\{x,y\}=\{z,w\}$ or not, $\{u,v,x,y,z,w\}$ forms an isolated copy of $K_4$ or $D_6$, and the rest of the graph is extremal. This contradicts the fact that $\Gamma$ is not extremal.

    \paragraph{Case 7.} One of $\{x,y\}$ and $\{z,w\}$ appears in at least two edges of $M$.\\
    Without loss of generality we may suppose that $xp,yq \in E(M)$ with $p \neq y$ and  $q \neq x$. We know from \cref{cl:deletevertex} that
    \[
    \mis(\Gamma-u-N(u)),\ \mis(\Gamma-v-N(v)) \leq \frac{9}{32}a(m,n).
    \]

    If $N(x) = \{u,v,p\}$, then  we know that $p$ cannot be adjacent to $u$ (otherwise we have $p = y$). Apart from $x$ and $v$, $p$ is adjacent to another vertex. Now in $\Gamma-u-v$ we have  $\deg(x) = 1$ and $\deg(p) \geq 2$, hence we may apply Case 2 to this case. Note that every vertex in $M$ has degree $3$, hence  equality in Case 2 cannot be attained. Therefore,
    \[
    \mis(\Gamma-v;\overline u) \leq \mis(\Gamma-u-v) < \frac{7}{8}a(m-1,n-2) = \frac{7}{16}a(m,n).
    \]
    The same argument works when $N(y) = \{u,v,q\}$.

    If $N(x) \neq \{u,v,p\}$ and $N(y) \neq \{u,v,q\}$, then applying~\cref{lem:basic-mis-ineq}(ii) and~\cref{cl:deletevertex}  to $\Gamma-u-v$ gives
    \[
    \mis(\Gamma-u-v;x), \ \mis(\Gamma-u-v;y)  \leq \frac{3}{8}a(m-1,n-2) = \frac{3}{16}a(m,n).
    \]
    Together with \cref{lem:basic-mis-ineq}(iv), we have
    \[
    \mis(\Gamma-v;\overline u) \leq \mis(\Gamma-u-v;x) + \mis(\Gamma-u-v;y) \leq \frac{3}{8}a(m,n).
    \]
    In both cases, by \cref{lem:basic-mis-ineq}(v) we have
    \[
    \mis(\Gamma) < \frac{9}{32}a(m,n) + \frac{9}{32}a(m,n) + \frac{7}{16}a(m,n) = a(m,n),
    \]
    as desired. This completes the proof of~\cref{thm:biggermatchingnumber}.

\subsection{Three stability extensions}
\begin{lemma}\label{lem:degree4}
    Let $\Gamma$ be an $n$-vertex graph, and $M$ be a matching of $m$ edges of $\Gamma$. Suppose that $V_1,\dots,V_k$ are pairwise disjoint vertex sets, each containing at least one edge from $M$, and for each $i$, every vertex in $V(M) \cap V_i$ has degree at least $4$ within $V_i$. Then,
    \[
    \mis(\Gamma) \leq a(k,m,n) := \left(\frac{59}{64}\right)^{k}2^{3m-n}3^{n-2m}.
    \]
\end{lemma}

\begin{proof}
    Suppose $\Gamma$ is a counterexample with minimum number of vertices.
    As \cref{thm:matching} implies the case $k=0$, we may assume $k\ge 1$.
    Choose $u \sim v$ in $M \cap V_1$. Recall that $v$ has $d\ge 4 $ neighbors in $V_1$. 

    Similarly to \cref{cl:deletevertex}, after deleting $v$ and the neighborhood of $v$  in $V_1$, $\Gamma-v-(N(v) \cap V_1)$ is an $(n-1-d)$-vertex graph, and $\{v\} \cup (N(v) \cap V_1)$ appears in at most $d$ edges in $M$. Furthermore, this deletion does not affect $V(M)\cap V_i$ for any $2\le i\le k$. Hence by \cref{lem:basic-mis-ineq}(i) and the minimality of $\Gamma$ we obtain
    \[
    \begin{split}
        \mis(\Gamma-v-N(v)) &\leq \mis(\Gamma-v-(N(v) \cap V_1)) \leq a(k-1,m-d,n-d-1)\\
        &= \frac{64\cdot3^{d-1}}{59\cdot2^{2d-1}}a(k,m,n) \leq \frac{27}{118}a(k,m,n).
    \end{split}
    \]
    Similarly, $\mis(\Gamma-u-N(u)) \leq 27a(k,m,n)/118$. By~\cref{lem:basic-mis-ineq}(iii), we have
    \[
    \mis(\Gamma-v;\overline u) \leq \mis(\Gamma-u-v) \leq a(k-1,m-1,n-2) =\frac{32}{59}a(k,m,n).
    \]
    Finally, we get from~\cref{lem:basic-mis-ineq}(v) that
    \[
    \begin{split}
    \mis(\Gamma) &\leq \mis(\Gamma-u-N(u)) + \mis(\Gamma-v-N(v)) + \mis(\Gamma-v;\overline u)\\
    &\leq \frac{27}{118}a(k,m,n) + \frac{27}{118}a(k,m,n) + \frac{32}{59}a(k,m,n) = a(k,m,n),
    \end{split}
    \]
    as desired.
\end{proof}

Denote by $T_{\ell}$ the graph consisting of two vertex disjoint $C_\ell$'s, $u_1 \cdots u_\ell$ and $v_1 \cdots v_\ell$, with a perfect matching $u_iv_i,1\leq i\leq \ell$ between them.

\begin{lemma}\label{lem:tl}
    Let $\ell\geq3$ and $\Gamma$ be an $n$-vertex graph. Suppose $\Gamma$ has $k$ copies of $T_\ell$'s and $m$ edges, all of which are vertex disjoint. Then,
    \[
    \mis(\Gamma) \leq b(k,m,n) := \left(\frac{31}{32}\right)^{k}2^{3(m+\ell k)-n}3^{n-2(m+\ell k)}.
    \]
\end{lemma}

\begin{proof}
Suppose to the contrary that there is a counterexample $(\Gamma,\cT,M)$ with minimum $k = |\cT|$, where $\cT$ is the collection of $T_{\ell}$ and $M$ is the collection of vertex disjoint edges, which are also vertex disjoint from $\cT$.

By \cref{thm:matching}, we may assume that $k \geq 1$. Choose a  $T \in \cT$ to be a copy of $T_{\ell}$, and let $(u_1,u_2,\dots,u_{\ell})$ and $(v_1,v_2,\dots,v_{\ell})$ be the two $C_{\ell}$'s of $T$, with $u_i \sim v_i$ for all $1 \leq i \leq \ell$. By~\cref{lem:basic-mis-ineq}(v), we have
 $$\mathrm{mis}(\Gamma)\le \mathrm{mis}(\Gamma-v_2-N(v_2))+\mathrm{mis}(\Gamma-u_2-N(u_2))+\mathrm{mis}(\Gamma-v_2;\overline{u_2}).$$
 
    We proceed to bound each of the terms above. After deleting $u_1,u_2,u_3,v_2$, we see that $(\cT-T,M+u_4v_4+\ldots+u_{\ell}v_{\ell})$ is a disjoint union of $T_{\ell}$'s and edges in $\Gamma-u_1-u_2-u_3-v_2$, which satisfies the inequality in \Cref{lem:tl} due to the minimality of $k$. Hence by~\cref{lem:basic-mis-ineq}(i), we have
    \[
    \mis(\Gamma-u_2-N(u_2)) \leq \mis(\Gamma-u_1-u_2-u_3-v_2) \leq b(k-1,m+\ell-3,n-4) \le \frac{9}{31}b(k,m,n).
    \]
    Similarly, we have $\mis(\Gamma-v_2-N(v_2)) \leq 9b(k,m,n)/31$. 

    For $\mis(\Gamma-v_2;\overline{u_2})$, if $\deg(u_2) = 3$, from \cref{lem:basic-mis-ineq}(iv),~(ii) and~(i) and the minimality of $k$, we have
    \[
    \begin{split}
    \mis(\Gamma-v_2;\overline{u_2}) &\leq \mis(\Gamma-v_2-u_2;u_1) + \mis(\Gamma-v_2-u_2;u_3)\\
    &\leq \mis(\Gamma-v_2-u_2-u_1-v_1-u_\ell) + \mis(\Gamma-v_2-u_2-u_3-v_3-u_4)\\
    &\leq b(k-1,m+\ell-3,n-5) + b(k-1,m+\ell-3,n-5) \leq \frac{12}{31}b(k,m,n).
    \end{split}
    \]
    Putting together, we have $\mis(\Gamma) \leq 30b(k,m,n)/31$. The same method works if $\deg(v_2) = 3$.

    The remaining case is when  $\deg(u_2),\deg(v_2) \geq 4$. Let $x$ be a neighbor of $u_2$ other than $u_1,u_3,v_2$. Classifying by $x \in V(T)\cup V(M),x \in V(\cT) \setminus V(T)$ or otherwise, we can bound $\mis(\Gamma-u_1-u_2-u_3-v_2-x)$, using the minimality of $k$, in each of these cases by $b(k-1,m+\ell-4,n-5)$, $b(k-2,m+2\ell-4,n-5)$ and $b(k-1,m+\ell-3,n-5)$ respectively; among these three terms, the maximum is $b(k-2,m+2\ell-4,n-5)$. Hence, we have
    \[
    \begin{split}
     \mis(\Gamma-u_2-N(u_2)) &\leq \mis(\Gamma-u_1-u_2-u_3-v_2-x)\\
    &\leq b(k-2,m+2\ell-4,n-5)= \frac{216}{961}b(k,m,n).
    \end{split}
    \]
    Similarly, $\mis(\Gamma-v_2-N(v_2)) \leq 216b(k,m,n)/961$,  and
    \[
    \mis(\Gamma-v_2;\overline{u_2}) \leq \mis(\Gamma-u_2-v_2) \leq a(k-1,m+\ell-1,n-2) = \frac{16}{31}b(k,m,n).
    \]
    Altogether we have $\mis(\Gamma) \leq 928b(k,m,n)/961$. This completes the proof.
\end{proof}

\begin{lemma}\label{lem:c4}
    Let $\Gamma$ be an $n$-vertex graph. Suppose $\Gamma$ has $k$ induced copies of $C_4$ and $m$ edges, all of which are vertex disjoint. Then,
    \[
    \mis(\Gamma) \leq c(k,m,n) := \left(\frac{4}{7}\right)^k2^{6k+3m-n}3^{n-2m-4k}.
    \]
\end{lemma}

\begin{proof}

    For  $U, V\subset \Gamma$, we denote by $\mis(\Gamma;U;\overline{V})$ the number of maximal independent sets $I$ satisfying  $I\cap U\ne \varnothing$ and $I\cap V=\varnothing$. We also write $\mis(\Gamma;U)=\mis(\Gamma;U;\varnothing)$,  $\mis(\Gamma;\overline{V})=\mis(\Gamma;V(\Gamma);\overline{V})$, and $\mis(\Gamma;U, W;\overline{V})$ for the number of maximal independent sets $I$ satisfying  $I\cap U\ne \varnothing, I\cap W\ne \varnothing$ and $I\cap V=\varnothing$.

    Suppose for a contradiction that there is a counterexample $(\Gamma,\cC,M)$ with minimum $k = |\cC|$, where $\cC$ is a  collection of pairwise vertex disjoint  induced $C_4$'s and $M$ is a matching, which is also vertex disjoint from $\cC$. By \cref{thm:matching}, we may assume that $k \geq 1$. Choose  $C = u_1u_2v_2v_1 \in \cC$, an   induced $C_4$, where $u_1u_2,u_2v_2,v_2v_1,v_1u_1$ are the edges.

    For each pair of $i,j\in\{1,2\}$, we construct a graph $\Gamma_{i,j}$ as follows: starting from $\Gamma$, we change the neighborhood of $u_1,u_2,v_2,v_1$ in $V(\Gamma-C)$ by letting $N_{\Gamma_{i,j}}(u_1) = N_{\Gamma_{i,j}}(u_2) = N_\Gamma(u_i)$ and $N_{\Gamma_{i,j}}(v_1) = N_{\Gamma_{i,j}}(v_2) = N_\Gamma(v_j)$. We shall prove that
    \begin{equation}\label{eq:gammaij}
    \mis(\Gamma) \leq \frac{2}{7}\sum_{i,j\in\{1,2\}}\mis(\Gamma_{i,j}).
    \end{equation}
    We show that each $I \in \MIS(\Gamma)$ corresponds to at least 3.5 maximal independent sets on the right hand side, by classifying the maximal independent sets by their intersections with $N(u_1)$, $N(u_2)$, $N(v_1)$ and $N(v_2)$. Denote $I'=I-u_1-u_2-v_1-v_2$.
    
    (i) When none of $I' \cap N(u_1),I' \cap N(u_2),I' \cap N(v_1),I' \cap N(v_2)$ is an empty set, then $I' = I$, and it is also a maximal independent set in each of $\Gamma_{i,j}$ for $i,j\in\{1,2\}$. The ratio is $1:4$.

    In the rest of the proof,  we  discuss only one case from each class of similar cases.

    (ii) When $I' \cap N(u_1) = \varnothing$ and all other intersections are non-empty, we have that $I=I'+u_1$; while $I'+u_1,I'+u_2$ are two maximal independent sets in each of  $\Gamma_{1,1}$ and $\Gamma_{1,2}$, and $I$ is a maximal independent set in $\Gamma_{2,1}$ and $\Gamma_{2,2}$. The ratio is $1:6$.

    (iii) When $I' \cap N(u_1)  = \varnothing $ and  $I' \cap N(u_2) = \varnothing$, and all other intersections are non-empty, we have that $I'+u_1,I'+u_2$ are two maximal independent sets in  $\Gamma$; while
    \[
    I'+u_1,I'+u_2 \in \MIS(\Gamma_{1,1}), \ \MIS(\Gamma_{1,2}),\  \MIS(\Gamma_{2,1}), \ \MIS(\Gamma_{2,2}),
    \]
    altogether the ratio is $2:8$.

    (iv) When $I' \cap N(u_1)=I'\cap N(v_1) = \varnothing$ and all other intersections are non-empty, we have that $I'+u_1,\ I'+v_1$ are two maximal independent sets in $\Gamma$; while
    \begin{gather*}
        I'+u_1+v_2,\ I'+u_2+v_1 \in \MIS(\Gamma_{1,1}),\quad I'+u_1,\ I'+u_2 \in \MIS(\Gamma_{1,2}),\\
        I'+v_1,\ I'+v_2 \in \MIS(\Gamma_{2,1}),\quad I' \in \MIS(\Gamma_{2,2}),
    \end{gather*}
    altogether the ratio is $2:7$.

    (v) When $I' \cap N(u_1)=I'\cap N(v_2) = \varnothing$ and all other intersections are non-empty, we have that $I=I'+u_1+v_2$; while
    \begin{gather*}
        I'+u_1+v_2, \ I'+u_2+v_1 \in \MIS(\Gamma_{1,2}),\quad I'+u_1,I'+u_2 \in \MIS(\Gamma_{1,1}),\\
        I'+v_1,\ I'+v_2 \in \MIS(\Gamma_{2,2}),\quad I' \in \MIS(\Gamma_{2,1}),
    \end{gather*}
    altogether the ratio is $1:7$.

    (vi) When $I' \cap N(u_1) \neq \varnothing$ and all other intersections are empty, we have that $I'+v_2,I'+u_2+v_1$ are two maximal independent sets in $\Gamma$; while
    \[
    I'+v_1,I'+v_2 \in \MIS(\Gamma_{1,1}),\MIS(\Gamma_{1,2}),\quad I'+u_1+v_2,I'+u_2+v_1 \in \MIS(\Gamma_{2,1}),\MIS(\Gamma_{2,2}),
    \]
    altogether the ratio is $2:8$.

    (vii) When each of the four intersections is $\varnothing$, then $I$ is also a maximal independent set in each  of $\Gamma_{i,j}$ for $i,j\in\{1,2\}$. Altogether the ratio is $1:4$.

    Combining all the above cases, \eqref{eq:gammaij} is proved, and by the pigeonhole principle there exists $i,j\in\{1,2\}$ with
    \begin{equation}\label{eq:Gammaij}
    \mis(\Gamma_{i,j}) \geq \frac{7}{8}\mis(\Gamma).
    \end{equation}

    We let  $\Gamma'=\Gamma-C$. Then, $(\cC-C,M)$ is a collection of induced $C_4$'s and edges of $\Gamma'$, which satisfies, following from the minimality of $k$ that
    \begin{equation}\label{eq:Gamma'}
    \mis(\Gamma') \leq c(k-1,m,n-4) = \frac{7}{16}c(k,m,n).
    \end{equation}

    We may partition $\MIS(\Gamma')$ into four classes, according to the intersection pattern of a maximal independent set with $N(u_i)$  and $ N(v_j)$ (where again, these neighborhoods are in $V(\Gamma)-C$ spanned by the graph $\Gamma$):
    \begin{align}
    \mis(\Gamma') = & \ \mis(\Gamma';N(u_i), N(v_j)) +  \mis(\Gamma';N(u_i);\overline{N(v_j)}) \notag \\  + &  \ \mis(\Gamma';N(v_j);\overline{N(u_i)}) + \mis(\Gamma';\overline{N(u_i) \cup N(v_j)}).
    \end{align}
    For $\MIS(\Gamma_{i,j})$ with $i,j$ as in $\eqref{eq:Gammaij}$, and the neighborhoods considered in $V(\Gamma)-C$, each independent set $I'\in \MIS(\Gamma';N(u_i), N(v_j))$ corresponds to one independent set $I'\in\MIS(\Gamma_{i,j};N(u_i),N(v_j))$. Each independent set $I'\in\MIS(\Gamma';N(u_i); \overline{N(v_j)})$ corresponds to two independent sets $I'+v_1,I'+v_2\in\MIS(\Gamma_{i,j};N(u_i); \overline{N_{\Gamma_{i,j}}(v_j)})$ and each independent set $I'\in \MIS(\Gamma';\overline{N(u_i)\cup N(v_j)})$ corresponds to two independent sets $I'+u_1+v_2,I'+u_2+v_1\in \MIS(\Gamma_{i,j};\overline{N(u_i)\cup N(v_j)})$. Therefore, we have

    \[
    \begin{split}
    \mis(\Gamma_{i,j}) =& \mis(\Gamma';N(u_i), N(v_j)) + 2\mis(\Gamma';N(u_i);\overline{N(v_j)}) \\&+ 2\mis(\Gamma';N(v_j);\overline{N(u_i)}) + 2\mis(\Gamma';\overline{N(u_i) \cup N(v_j)}).
    \end{split}
    \]
    Therefore, $\mis(\Gamma_{i,j})\leq 2\mis(\Gamma')$ and together with \eqref{eq:Gammaij} and \eqref{eq:Gamma'} it follows
    \[
    \mis(\Gamma) \leq \frac{8}{7}\mis(\Gamma_{i,j}) \leq \frac{16}{7}\mis(\Gamma') \leq c(k,m,n),
    \]
    as desired.
\end{proof}

\section{Preliminaries}\label{sec:prelim}
We list in this section some tools that will be needed.
The first is a stability result on large sum-free sets for type \ri\ groups from Green and Ruzsa \cite{2005Sum}.

\begin{lemma}[\cite{2005Sum}]\label{lem:stability}
    Suppose that $G$ is a type \ri($p$) group of order $n$, with $p = 3k + 2$. Let $A \subseteq G$ be maximal sum-free, and suppose that $|A| > (p+2)n/3(p+1)$. Then there is  a homomorphism $\varphi\colon G \to \Z_p$ such that $A = \varphi^{-1}(\{k+1,\dots,2k+1\})$.
\end{lemma}

The following simple observation is also needed.

\begin{observation}\label{observation:solutions}
    Let $G = \Z_{2^{\alpha_1}} \oplus \ldots \oplus \Z_{2^{\alpha_r}} \oplus K$ be an even order group such that $r,\alpha_1,\dots,\alpha_r\in\N$ and $K$ is odd. Let $\varphi\colon G\to\Z_2$ be a homomorphism. Then the equation $2x = 0$ has exactly $2^r$ solutions, with at most half of them satisfying $\varphi(x) = 1$. Furthermore, if $2x' = s$ for some $s \in G$ and $\varphi(x') = 1$, then $2x = s$ has at least $2^{r-1}$ solutions with $\varphi(x) = 1$.
\end{observation}

\begin{proof}
    Let $x_i=\left(0,\dots,0,2^{\alpha_{i}-1},0,\dots,0\right)$, where the only non-zero coordinate is the $i$-th coordinate, for $1 \leq i \leq r$. Then, $S_0 \coloneqq \spa\{x_1,\dots,x_r\}$ is the set of solutions to $2x = 0$, which has size $2^r$.

    If $\varphi(S_0) = \{0\}$, then all the solutions satisfy $\varphi(x) = 0$. Otherwise, there exists a nonzero $x \in S_0$ with $\varphi(x) = 1$. For each $y \in S_0$ with $\varphi(y) = 1$, we have $2(x+y) = 2x+2y = 0$ and hence $x+y \in S_0$. Moreover, $\varphi(x+y) = \varphi(x)+\varphi(y) = 0$. Hence at most half of the elements in $S_0$ satisfy $\varphi(x) = 1$.

    Now suppose $2x' = s$ with $\varphi(x') = 1$. Then, $S' \coloneqq x'+S_0$ is the set of solutions to $2x = s$, which also has size $2^r$. Moreover, for each $y \in S_0$ with $\varphi(y) = 0$, we have $x'+y \in S'$ and $\varphi(x'+y) = 1$. Hence at least half of the elements in $S'$ satisfy $\varphi(x) = 1$.
\end{proof}

\subsection{Containers and link graphs}
The analysis of $f_{\max}(G)$ and $f^*_{\max}(G)$ involves the application of the hypergraph container method. The following lemma presented in \cite{hassler2025notessumfreesetsAbelian} is an analogue of a result of Green and Ruzsa \cite{2005Sum}. The hypergraph container method~\cite{BMS, ST} gives better quantitative bounds, but they are not needed for our proofs. 

\begin{lemma}[\cite{2005Sum},\cite{hassler2025notessumfreesetsAbelian}]\label{lem:container}
    For every $\delta > 0$ there exists an $n_0 \in \N$ such that the following holds. If $G$ is an Abelian group of order $n \geq n_0$, then there is a container family $\cF$ of subsets of $G$ with the following properties.
   
        (i) $|\cF| \leq 2^{\delta n}$.
        
        (ii) If $I \subseteq G$ is distinct sum-free, then $I$ is contained in some $F \in \cF$.
        
        (iii) Every $F \in \cF$ is of the form $F = A \cup B$, where $A$ is maximal sum-free and $|B| \leq \delta n$.
\end{lemma}

Note that every sum-free set is distinct sum-free, 
hence a sum-free set  $I$ is also contained in some $F \in \mathcal{F}$.

\begin{defn}\label{def:link}
    For disjoint subsets $A,S \subseteq G$, let $L_S[A]$ be the \textit{link graph of $S$ on $A$} defined as follows. Its vertex set is $A$ and its edge set consists of the following edges:
    
    (i) two distinct vertices $x,y \in A$ are adjacent if there exists $s \in S$ such that $\{x,y,s\}$ is a Schur triple;
        
    (ii) there is a loop at a vertex $x \in A$ if there exist distinct $s,t\in S$ such that $\{x,s,t\}$ is a Schur triple.

(iii) there is a loop at a vertex $x \in A$ if there exist  $s\in S$ such that $\{x,x,s\}$ or $\{x,s,s\}$ is a Schur triple.
\end{defn}

Note that although \cref{thm:matching} is stated for simple graphs, it also works on graphs with loops, by simply deleting the vertices having loops.

We also need the \textit{distinct link graph} $L_S^*[A]$ defined as in~\cref{def:link} but using only edges from (i) and (ii). 

\begin{lemma}\label{lem:distinct-more}
 For any disjoint subsets $A,S$ in an Abelian group $G$, 
 $$\mis(L_S[A]) \leq \mis\left(L_S^*[A]\right).$$
\end{lemma}
\begin{proof}
    Write $\Gamma=L_S[A]$ and $\Gamma^*=L_S^*[A]$. Note that if $(x,y,s)$ is a Schur triple with $x \neq y \in A$ and $s \in S$, then it is a distinct triple since $s \notin A$. Hence, the only difference  between $\Gamma$  and $\Gamma^*$ is that $\Gamma$ might have  some  loops of type (iii) in \cref{def:link}.

    Let $X$ be the set of vertices with loops in $\Gamma$. Then by \cref{lem:basic-mis-ineq}(i), we have
    \[
    \mis(\Gamma) = \mis(\Gamma-X) = \mis\left(\Gamma^*-X\right) \leq \mis\left(\Gamma^*\right),
    \]
    completing the proof.
\end{proof}

The following lemma reduces the problem of bounding the number of maximal sum-free sets to a problem on maximal independent sets in the link graph. For subsets $S\subseteq F$ of a group, we use the symbol $\MSF(F)$ ($\MSF^*(F)$) for the set of the maximal (distinct) sum-free sets contained in $F$, and $\MSF(F;S)$ ($\MSF^*(F;S)$) for the family of such sets containing $S$. Denote by $\msf(\cdot)$  the cardinality of  $\MSF(\cdot)$.

\begin{lemma}\label{lem:mis-reduction}
    For every $\alpha,\delta > 0$ there exists an $n_0 \in \N$ such that the following holds. Let $G$ be an Abelian group of order $n \geq n_0$. If for each pair of $A,S\subseteq G$ such that $A$ is maximal (distinct) sum-free, $S$ is (distinct) sum-free and $A \cap S = \varnothing$, 
    $$\mis(L_S^{(*)}[A]) \leq 2^{\alpha n},$$
    then 
    $$f_{\max}^{(*)}(G)\le 2^{\alpha n+2\delta n}.$$
\end{lemma}

\begin{proof}
Let us first bound $f_{\max}(G)$. Apply \cref{lem:container} to $G$ with $\delta$ to obtain a family $\cF$. Then, by~\cref{lem:container}(ii) and (iii), we have
    \[
    \MSF(G) = \bigcup_{F\in\cF}\MSF(F) \subseteq  \bigcup_{F\in\cF}\bigcup_{\substack{S\subseteq B\\S\ \text{sum-free}}}\MSF(A\cup S;S).
    \]
    Indeed, every maximal sum-free set $I$ contained in $F=A\cup B$ can be constructed by first choosing a sum-free $S\subseteq B$ and then extending $S$ in $A$ to a maximal one.

    Now, fix $F=A\cup B\in\cF$ and a sum-free $S\subseteq B$. By~\cref{lem:container}(i) and (iii), there are at most $2^{2\delta n}$ such choices. To prove $f_{\max}(G)\le 2^{\alpha n+2\delta n}$, it suffices to bound $\msf(A\cup S; S)$ by $2^{\alpha n}$. 

    To upper bound $\msf(A\cup S; S)$, construct the link graph $\Gamma = L_S[A]$. Notice that each maximal sum-free set in $\MSF(A\cup S;S)$ corresponds to a maximal independent set in $\Gamma$. Thus, by assumption
    \[
    \msf(A\cup S;S) \le  \mis\left(\Gamma\right)\le 2^{\alpha n}.
    \]
    The proof for $f_{\max}^{*}(G)$ is almost identical by considering instead the distinct link graph $L_S^{*}[A]$.
    \end{proof}


\section{Infinitely many groups with intermediate growth}\label{sec:growth}

In this section, we prove~\cref{thm:growth} with the following family of Abelian groups of  type \ri(23). Note that by \cref{thm:mu}(i), we have below $\mu(G)=8n/23$, and one can check that the ``$\ll$" on both sides are valid.

\begin{theorem}\label{thm:intermediate-growth}
    For each prime $p \geq 29$, let $G = \Z_3^2 \oplus \Z_{23} \oplus \Z_p$ and $n = |G| = 207p$.  Then one has
    \[
    2^{\mu(G)/2} \ll 3^{n/9} \leq f_{\max}(G) \leq 3^{25n/216+o(n)} \ll  3^{\mu(G)/3}.
    \]
\end{theorem}

\begin{proof}
Let $\delta>0$. We shall prove that $3^{n/9}\le f_{\max}(G)\le 3^{25n/216+2\delta n}$. 

Let us first establish the lower bound.  Let $s= (0,1,0,0)$. We shall make use of subsets of the set $A = \{1\} \oplus \Z_3 \oplus \Z_{23} \oplus \Z_p$. For each function $f\colon\Z_{23} \oplus \Z_p \to \Z_3$, denote $A_f = \{(1,f(x),x):x \in \Z_{23} \oplus \Z_p\} \cup \{s\}$. It is not hard to see that each $A_f$ is sum-free, and can be extended to a maximal sum-free set. Moreover, for functions $f \neq g$, $A_f \cup A_g$ is not sum-free. Indeed, without loss of generality we may assume that $g(x) = f(x) + 1$ for some $x \in \Z_{23} \oplus \Z_p$, then $(1,f(x),x),(1,f(x)+1,x),s \in A_f\cup A_g$, with $(1,f(x),x)+s=(1,f(x)+1,x)$. Hence $A_f$ and $A_g$ cannot extend to the same maximal sum-free set. Consequently, $f_{\max}(G)$ is at least the  number of such functions $f$, which is $3^{n/9}$.

\smallskip

To obtain the upper bound, fix $A,S\subseteq G$ such that $A$ is maximal sum-free, $S$ is sum-free and $A \cap S = \varnothing$. Construct the link graph $\Gamma = L_S[A]$. Then by~\cref{lem:mis-reduction}, it suffices to prove that 
$$\mis(\Gamma) \leq 3^{25n/216}.$$

    If $|A| \leq 25n/72$, then Moon-Moser Theorem tells us that $\mis(\Gamma) \leq 3^{25n/216}$, as desired.

    We may then assume that $|A| > 25n/72$. Apply \cref{lem:stability} to find a homomorphism $\varphi\colon G \to \Z_{23}$ with $A = \varphi^{-1}(M)$, where 
    $$M=\{8,9,10,11,12,13,14,15\}\subseteq \Z_{23}.$$ 
    Thus, $|A|=8n/23$. If $S$ is an  empty set, then $\Gamma$ is an edgeless graph and $\mis(\Gamma) = 1$, as desired. Otherwise, choose some $s \in  S$.

    If $\varphi(s) = 0$, then as $0-M=M$ in $\Z_{23}$, we see that for each $x \in A$, $s-x$ is also in $A$. Moreover, since $0\notin 2M$, there is no solution to $2x = s$ in $A$. Hence $\Gamma$ contains a perfect matching that matches $x$ and $s-x$, then by~\cref{thm:matching}, we have
    \[
    \mis(\Gamma) \leq 2^{|A|/2} = 2^{4n/23} < 3^{25n/216}.
    \]

    If $\varphi(s) \in \pm\{1,2,3,4\}$, then  for each $x \in A$, at least one of $\{x+s,x-s\}$ is in $A$. Denote
    \[
    P_s(x) = A \cap \{x,x+s,\dots,x+22s\}.
    \]
    Notice that as $23$ is a prime, $A$ has the decomposition
    \[
    A = \bigsqcup_{\varphi(x)=8}P_s(x).
    \]
    We construct a large matching as follows. For each $x$ with $\varphi(x) = 8$, the set $P_s(x)$ contains a number of vertex disjoint paths of the form $\{x+a\cdot s,\dots, x+b\cdot s\}$ with $a,b\in\Z_{23}$ in $L_S[A]$, and since at least one of $\{x-s,x+s\}$ is in $A$, we see that $a\neq b$, i.e.,~each path has length at least $1$. As $P_s(x)$ has 8 vertices and contains a linear forest with no isolated vertex, it has a matching of size $3$. Therefore, $\nu(\Gamma) \geq 3|A|/8$, and by \cref{thm:matching},
    \[
    \mis(\Gamma) \leq 2^{|A|/8}3^{|A|/4} = 2^{n/23}3^{2n/23} < 3^{25n/216}.
    \]
    
    We may now assume that for every $s\in S$, $\varphi(s) \in \pm\{5,6,7\}$. For each $x \in A$ with $\varphi(x) \in \{11,12\}$, we claim that $x+s,x-s,s-x$ are not in $A$.  Indeed, recall that $A=\varphi^{-1}(M)$, while for any $s\in S$,
    \[
    \varphi(x+s),\varphi(x-s),\varphi(s-x) \in \{4,5,6,7,16,17,18,19\}\subseteq \Z_{23}\setminus M.
    \]
    Hence $x$ is isolated in $\Gamma$, and by the Moon-Moser Theorem,
    \[
    \mis(\Gamma) \leq 3^{(3|A|/4)/3} = 3^{2n/23} < 3^{25n/216}.
    \]
    This completes the proof.
\end{proof}

\section{Maximal distinct sum-free sets}\label{sec:distinct}

In this section we establish an upper bound for $f_{\max}(G)$ and $f^*_{\max}(G)$ for even-order groups. \cref{thm:even-distinct} follows from the following theorem and~\cref{prop:counter1,prop:counter2}.

\begin{theorem}\label{thm:even*}
    Let $G$ be an even-order $n$-element group, which is not $\Z_2^k \oplus \Z_3$ for any integer $k$. Then
    \[
    f_{\max}(G),f^*_{\max}(G) \leq 2^{n/4 + o(n)}.
    \]
\end{theorem}

\begin{proof}
    We first present the proof for $f^*_{\max}(G)$. 
    We consider the distinct link graph $\Gamma^*:=L_S^*[A]$. Since $A$ is sum-free, we have $|A| \leq \mu(G) = n/2$. By~\cref{lem:mis-reduction}, it suffices to prove that for\ each pair of $A$ and $S$ such that $A$ is maximal distinct sum-free, $S$ is distinct sum-free and $A \cap S = \varnothing$, one has
    \[
    \mis(\Gamma^*) \leq 2^{n/4}.
    \]
    If $|A| \leq 4n/9$, then by the Moon-Moser Theorem, $\mis\left(\Gamma^*\right) \leq 3^{4n/27} < 2^{n/4}$, as desired.

    We may then assume that $|A| > 4n/9$. Apply \cref{lem:stability} to find a homomorphism $\varphi\colon G \to \Z_2$ with $A = \varphi^{-1}(1)$, and thus $|A| = n/2$ and  $S \subseteq \varphi^{-1}(0)$. If $S$ is empty, then $\Gamma^*$ is an edgeless graph and thus $\mis\left(\Gamma^*\right) = 1$, as desired. We now assume that $S$ is non-empty.

    If $S=\{0\}$, then $\Gamma^*$ consists of a matching and some isolated vertices, which also gives $\mis\left(\Gamma^*\right) \leq 2^{|A|/2} = 2^{n/4}$. Otherwise, pick $0 \neq s \in  S$ with minimal order.

    If $s$ has order $|s| = 2$, then $\Gamma^*$ admits a perfect matching $\{x,x+s\}_{x\in A}$. Directly applying \cref{thm:matching} we obtain $\mis\left(\Gamma^*\right) \leq 2^{|A|/2}$.

    Now we suppose that all elements in $S$ have order at least 3, and we first discuss the cases when $S\subseteq\{s,-s\}$. 

    Recall that $T_{\ell}$ is the graph consisting of two vertex disjoint $C_\ell$'s and a perfect matching between them. We denote by $T_{\ell}^{+}$ the graph consisting of two vertex disjoint $C_\ell$'s, $u_1 \cdots u_\ell$ and $v_1 \cdots v_\ell$, with two perfect matchings $u_iv_{\ell-i+1},u_iv_{\ell-i+3},1\leq i\leq \ell$ between them. Here, $i+1 = 1$ for $i=\ell$. If $-x \in \{x,x+s,\dots,x+(|s|-1)s\}$, then the component containing $x$ in $\Gamma^*$ is $\{x,x+s,\dots,x+(|s|-1)s\}$, constituting a $C_{|s|}$. If not, then the component is a $T_{|s|}$ or $T_{|s|}^+$, depending on $S=\{s\}$ or $S=\{s,-s\}$. Note that $T^+_{|s|}$ contains a copy of $T_{|s|}$ (simply reordering the vertices as $v_{\ell-i+1}\to v_i$).

    When $|s| = 3$, we count the $C_3$ components in $\Gamma^*$. Each $C_3$ component is of the form $\{x, x+s,x+2s\}$ with $-x\in \{x, x+s,x+2s\}$, implying  that there is a $k \in \Z_3$ with $-x = x + ks$. It follows that $2(x+(k/2)s) = 0$. Hence each $C_3$ component contains a solution to $2y = 0$. On the other hand, we know that $3$ divides $n/2^r$ as $G$ has an order $3$ element. Since the group is not $\Z_2^k \oplus \Z_3$, we know that $n/2^r \geq 6$ in the setting of~\cref{observation:solutions}. Hence $2y = 0$ has at most $n/12$ solutions in $A$, and the number of $C_3$ components in $\Gamma^*$ is at most $n/12$. We know that $\mis(T_3) = 6$, $\mis(T_3^+) = 3$, and $\mis(C_3) = 3$. Therefore,
    \begin{equation}\label{dis41}
    \mis\left(\Gamma^*\right) \leq 3^{n/12}6^{(|A|-n/4)/6} = 3^{n/12}6^{n/24} < 2^{n/4}.
    \end{equation}
    
    For $|s| \geq 4$, it is known from \cite[Example 1.2]{1987MISFuredi} that $\mis(C_{|s|}) < 1.4^{|s|} < 2^{|s|/2}$. Recall that $T_{|s|}\subseteq T^+_{|s|}$, so \cref{lem:tl} is valid for both $T_{|s|}$ and $T_{|s|}^+$ and it (applied with $k=1$, $m=0$, $n=2s$ and $\ell=s$) yields that both $\mis(T_{|s|})$ and $\mis(T_{|s|}^+)$ are at most $\frac{31}{32} \cdot 2^{|s|} < 2^{|s|}$. Hence $\mis\left(\Gamma^*\right) < 2^{|A|/2}$, as desired.

    Now we proceed to handle the case when $S \nsubseteq \{s,-s\}$.

    Let $t \in S$ with $t \notin \{s,-s\}$. For each $x$, let $V_x = \{x+is+jt\colon 0 \leq i \leq |s|-1,0 \leq j \leq |t|-1\}$. Then for each $x \in A$, $x+s,x+t,x-s,x-t$ are four distinct neighbors of $x$ in $V_x$. Furthermore, we know that
    \[
    A = \bigcup_{x \in A}V_x,
    \]
    where for each pair $x,y \in A$, either $V_x = V_y$ or $V_x \cap V_y = \varnothing$.
    
    Let $\ell = |V_x|$ (we know that all $V_x$ have the same size). Then $A = V(\Gamma)$ is partitioned into $n/2\ell$ disjoint vertex subsets $V_x$. We find in each $V_x$ an almost perfect matching (i.e., a matching of size $\lfloor\ell/2\rfloor$). Indeed, if $|s|$ or $|t|$ is even (without loss of generality suppose $|s|$ is even), then for each $0 \leq j \leq |t|-1$ and odd $0 \leq i \leq |s|-1$, we match $x+is+jt$ with $x+(i-1)s+jt$, which results in a perfect matching. If both $|s|$ and $|t|$ are odd, then for each $0 \leq j \leq |t|-1$ and odd $0 \leq i \leq |s|-1$, we match $x+is+jt$ with $x+(i-1)s+jt$; for each odd $0 \leq j \leq |t|-1$ and $i = |s|-1$, we match $x+is+jt$ with $x+is+(j-1)t$. Then all vertices are matched except $x+(|s|-1)s+(|t|-1)t$. Hence the matching has size $(\ell-1)/2$, which is almost perfect.
    
    Let $M$ be the matching of $\Gamma$ consisting of such matchings of $V_x$, then $M$ has size at least $n(\ell-1)/4\ell$. \cref{lem:degree4} gives that
    \[
    \mis\left(\Gamma^*\right) \leq 59^{n/2\ell}2^{n(\ell-15)/4\ell}3^{n/2\ell} < 2^{n/4}.
    \]
    This completes the proof for $f_{\max}^*(G)$.
    
    For $f_{\max}(G)$, we analogously consider the link graph $L_S[A]$, and the desired bound follows from~\cref{lem:distinct-more}. 
\end{proof}

\begin{remark}
    For the distinct sum-free sets, the lower bound was known in \cite{hassler2025notessumfreesetsAbelian} that $f^*_{\max}(G) \geq 2^{(n-2)/4}$.
\end{remark}

\subsection{Counterexamples}

\begin{proposition}\label{prop:counter1}
    Let $k\in\mathbb{N}$ and $n = 3 \cdot 2^k$. Then $f^*_{\max}\left(\Z_2^k \oplus \Z_3\right) = 3^{n/6+o(n)}$.
\end{proposition}

\begin{proof}
    The upper bound follows from the known facts that $\mu^*(\Z_2^k \oplus \Z_3)= n/2$, and that for every group $G$ we have $f^*_{\max}(G)\le 3^{\mu^*(G)/3+o(n)}$.
    
    For the lower bound, let $s= (0,\dots,0,1)$. We define subsets of $A = \{1\} \oplus \Z_2^{k-1} \oplus \Z_3$ as follows. For each function $f\colon\Z_2^{k-1} \to \Z_3$, denote $A_f = \left\{(1,x,f(x)):x \in \Z_2^{k-1}\right\} \cup \{s\}$. Then each $A_f$ is distinct sum-free and extends to a maximal distinct sum-free set. Moreover, for any functions $f \neq g$, $A_f \cup A_g$ is not distinct sum-free. Indeed, without loss of generality suppose $g(x) = f(x) + 1$ for some $x \in \Z_2^{k-1}$, then $(1,x,f(x)),(1,x,f(x)+1),s \in A_f\cup A_g$, with $(1,x,f(x))+s=(1,x,f(x)+1)$. Thus $A_f$ and $A_g$ cannot extend to the same maximal distinct sum-free set, and $f^*_{\max}(G)$ is at least the  number of such functions $f$, which is $3^{n/6}$.
\end{proof}

For odd groups, there are several ways to construct counterexamples. We present one of them.

\begin{proposition}\label{prop:counter2}
    For each prime $p \geq 23$ with $p \equiv 2\ (\text{mod}\ 3)$, one has
    \[
    f^*_{\max}\left(\Z_3^2 \oplus \Z_p\right) \geq 3^p \ge 2^{\mu(G)/2} = 2^{3(p+1)/2}.
    \]
\end{proposition}

\begin{proof}
    Let $s=(0,1,0)$ and define $A_f=\left\{(1,f(x),x):x \in \Z_p\right\} \cup \{s\}$, with $f\colon\Z_p \to \Z_3$. Similarly to the proof of~\cref{prop:counter1}, each $A_f$ is distinct sum-free and two sets $A_f$ and $A_g$ cannot extend to the same maximal distinct sum-free set. Thus we have $f^*_{\max}(G) \geq 3^{p}$.
\end{proof}

\section{All but one even-order group has few maximal sum-free sets}\label{sec:even-order}

In this section, we prove~\cref{thm:even}, i.e., that $f_{\max}(G)$ is exponentially smaller than $2^{n/4}$ for every  even group $G$ other than $\Z_2^k$. We need the following statement.

\begin{theorem}[\cite{2021GroupsLiu}]\label{thm:l-s}
    There exists a constant $c > 10^{-4}$ such that the following holds. For every $C\ge 10^{30}$, there is an integer $n_0=n_0(C)$ such that for every Abelian group $G = \Z_{2^{\alpha_1}} \oplus \ldots \oplus \Z_{2^{\alpha_r}} \oplus K$ with even order $n>n_0$, $|K|$ odd and $n/2^r\ge C$, we have
    \[
    f_{\max}(G) \leq 2^{(1/4-c)n}.
    \]
\end{theorem}

We shall prove the following counterpart of  \cref{thm:l-s}, from which~\cref{thm:even} follows.

\begin{theorem}\label{thm:even-counterpart}
   For each $C \geq 2$ there exists an integer $n_0$ such that the following holds. For every Abelian group
    $G = \Z_{2^{\alpha_1}} \oplus \ldots \oplus \Z_{2^{\alpha_r}} \oplus K$ with even order $n > n_0$, $|K|$ odd and $2 \leq n/2^r \leq C$, we have
    \[
    f_{\max}(G) \leq 2^{(1/4-c)n},
    \]
    where   $c = 1/600C^2$. 
\end{theorem}

\begin{proof}
    For a maximal sum-free set $A$ and a sum-free set $S$ with $A \cap S = \varnothing$, we construct the link graph $\Gamma = L_S[A]$. By~\cref{lem:mis-reduction}, it suffices to prove that for all such pairs $(A,S)$, one has
    \[
    \mis(\Gamma) \leq 2^{n/4-2cn}.
    \]
    We have  a case analysis similar to the one  in the proof of \cref{thm:even*}. Note that as $S$ is sum-free, $0 \notin S$.

    If $|A| \leq 4n/9$, then $\mis(\Gamma) \leq 3^{4n/27} < 2^{n/4-2cn}$ by the choice of $c$.
    
    We may assume that $|A| > 4n/9$. Apply \cref{lem:stability} to find a homomorphism $\varphi\colon G \to \Z_2$ with $A = \varphi^{-1}(1)$. Then, $|A|=n/2$. Pick $s \in S$ with minimal order. We know that if $G = \Z_{2^{\alpha_1}} \oplus \ldots \oplus \Z_{2^{\alpha_r}} \oplus K$ with $\alpha_1 \leq \ldots \leq \alpha_r$, then $s$ has order
    \[
    |s| \leq  2^{\alpha_r}\cdot |K|= \frac{n}{2^{\alpha_1}2^{\alpha_2}\cdot\ldots\cdot 2^{\alpha_{r-1}}} \leq\frac{n}{2^{r-1}} \le 2C.
    \]

    We first discuss the case when $|s| = 2$. For each $x \in A$, $\{x,x+s\}$ is an edge in $\Gamma$, and thus the matching number of $\Gamma$ is $|A|/2=n/4$.
    
    If $2x = s'$ has a solution in $A$ for some $s' \in S$, then by \Cref{observation:solutions}, it has at least $2^{r-1}$ solutions in $A$. Fix such an $s'$ and let $x_1,\dots,x_{2^{r-1}} \in A$ be solutions to $2x = s'$. Then they cannot appear in any independent set of $\Gamma$, since there are loops at them. Since $\Gamma-x_1-\ldots-x_{2^{r-1}}$ has matching number at least $n/4-2^{r-1}$, \cref{thm:matching}  gives the desired bound
    \[
    \begin{split}
    \mis(\Gamma) &= \mis(\Gamma-x_1-\dots-x_{2^{r-1}}) \leq 2^{3\left(n/4-2^{r-1}\right)-\left(n/2-2^{r-1}\right)}3^{\left(n/2-2^{r-1}\right)-2\left(n/4-2^{r-1}\right)}\\
    &= \left(\frac{3}{4}\right)^{2^{r-1}}2^{n/4} \le \left(\frac{3}{4}\right)^{n/2C}2^{n/4} < 2^{n/4-2cn}.
    \end{split}
    \]

    We can now assume that $2A \cap S = \varnothing$. If the two edges $\{x,x+s\}$ and $\{s-x,-x\}$ are vertex disjoint, then there is a matching between them (edges $\{x, s-x\}, \{-x, x+s\}$), and they together form a $C_4$ in $\Gamma$ (we argue later that it is actually an induced $C_4$). If they are not vertex disjoint, then  they coincide, and  it must be that $2x = 0$, since $x = s-x$ or $x+s = -x$ would imply $2A \cap S \neq \varnothing$.
    
    Denote $m$  the number of coincided edges. Recall that each such edge is of the form $\{x_i,x_i+s\}$ with both endpoints being a solution to the equation $2x = 0$. As there are at most $2^{r-1}$ solutions to $2x = 0$ in $A$ due to~\cref{observation:solutions} and $A=\varphi^{-1}(1)$, we have $m \leq 2^{r-2}\le\frac{n}{8}$, and hence the number of $C_4$'s is
    \[
    k := \frac{\frac{n}{2}-2\cdot \#\ \text{coincided edges}}{4} = \frac{\frac{n}{2}-2m}{4}=\frac{n}{8}-\frac{m}{2}\geq \frac{n}{16}.
    \]
    We claim that each such $4$-cycle $\{x,x+s,-x,s-x\}$ is an induced $C_4$. If not, then $\{x,-x\}$ or $\{s+x,s-x\}$ is an edge in $\Gamma$. Either $x+(-x) = 0 \in S$, or $\pm(x-(-x)) = \pm2x \in S$, contradicting the fact that $S$ is sum-free and $2A \cap S = \varnothing$. Thus by \cref{lem:c4} with $(k,m,n)_{\ref{lem:c4}}=(\frac{n}{8}-\frac{m}{2},m,\frac{n}{2})$, we have
    \[
    \mis(\Gamma)  \leq \left(\frac{4}{7}\right)^{k}2^{n/4}\le \left(\frac{4}{7}\right)^{n/16}2^{n/4} < 2^{n/4-2cn}.
    \]

    We next consider the case when $|s| > 2$, and $S \subseteq \{s,-s\}$. If $|s| = 3$ and $G \neq \Z_2^k \oplus \Z_3$, then we follow the proof of \cref{thm:even*} (see \eqref{dis41}) to obtain
    \[
    \mis(\Gamma) \leq 3^{n/12}6^{n/24} < 2^{n/4-2cn}.
    \]
    If $G = \Z_2^k \oplus \Z_3$, then $s = \pm(0,\dots,0,1)$, and $\Gamma$ is the graph consisting of $2^{k-1}$ isolated triangles $\{x,x+s,x+2s\}_{x\in A}$, each with a loop (at a vertex that corresponds to a solution to $2x=s$). We have
    \[
    \mis(\Gamma) = 2^{2^{k-1}} = 2^{n/6} < 2^{n/4-2cn}.
    \]
    If $4 \leq |s| \le 2C$, we have from the proof of   \cref{thm:even*} that the distinct link graph $\Gamma^*=L_S^*[A]$ is a disjoint union of $C_{|s|}$ and $T_{|s|}$ or $T_{|s|}^+$. Let the number of $C_{|s|}$ and $T_{|s|}$ (or $T_{|s|}^+$) be $\ell$ and $k$, respectively. Then we have
    \[
    \ell|s|+2k|s| = |A| = \frac{n}{2}.
    \]
    Hence, since $|s| \le 2C$ and recall that $\Gamma$ is a supergraph of $\Gamma^*$ with some additional loops, we see that
    \[
    \begin{split}
    \mis(\Gamma) &\leq \mis\left(\Gamma^*\right) < 1.4^{|s|\ell}31^k2^{|s|k-5k}\\
    &= \left(\frac{1.4^2}{2}\right)^{|s|\ell/2}\left(\frac{31}{32}\right)^{k}2^{|s|\ell/2+|s|k} < \left(\frac{31}{32}\right)^{n/4|s|}2^{n/4} \le \left(\frac{31}{32}\right)^{n/8C}2^{n/4} < 2^{n/4-2cn}.
    \end{split}
    \]
    
    It remains to handle the case that $S \not\subseteq \{s,-s\}$. We use the notation $t,V_x,\ell$ as in the proof of \cref{thm:even*}. Then
    \[
    \ell \leq |s|\cdot|t| < 2C \cdot 2C = 4C^2.
    \]
    Hence
    \[
    \mis(\Gamma) \leq \mis\left(\Gamma^*\right) \leq 59^{n/2\ell}2^{n(\ell-15)/4\ell}3^{n/2\ell} = \left(\frac{3^2 \cdot 59^2}{2^{15}}\right)^{n/4\ell}2^{n/4} < 0.96^{n/16C^2}2^{n/4} < 2^{n/4-2cn}.
    \]
    This completes the proof.
\end{proof}

\paragraph{Acknowledgement.}
The first and the third author are grateful for BIRS (The Institute for Advanced Study in Mathematics)  hosting  the``Finite Geometry and Extremal Combinatorics" workshop in Hangzhou, China, in June  2025, where they have initiated this project. When finishing the project, we learned that Shagnik Das and Tuan Tran have obtained a result  similar to~\cref{thm:matching}.

\bibliographystyle{abbrv}
\bibliography{bib}

\end{document}